\documentclass[12pt,leqno,amscd, amsfonts, amssymb,pstricks,verbatim]{amsart}
\usepackage{amsfonts}
\usepackage{amscd}
\usepackage{mathrsfs}
\usepackage{amsmath}
\usepackage{amssymb}
\usepackage{hyperref}
\usepackage{bookmark}

\def\a{\alpha}
\def\r{\gamma}
\def\b{\beta}

\def\Z{\mathbb{Z}}

\def\C{\mathbb{C}}

\numberwithin{equation}{section}
\newtheorem{theo}{Theorem}[section]
\newtheorem{defi}[theo]{Definition}

\newtheorem{lemm}[theo]{Lemma}
\newtheorem{prop}[theo]{Proposition}
\newtheorem{clai}{Claim}
\newtheorem{rema}[theo]{Remark}
\newtheorem{case}{Case}

\allowdisplaybreaks

\begin{document}

\title[Non-weight modules over the affine-Virasoro algebra of type $A_1$]{Non-weight modules over the affine-Virasoro algebra of type $A_1$}

\author{Qiu-Fan Chen and Jian-Zhi Han}

\address{Department of Mathematics, Shanghai Maritime University,
 Shanghai, 201306, China.}\email{chenqf@shmtu.edu.cn}
\address{School of Mathematical Sciences, Tongji University,
 Shanghai, 200092, China.}\email{jzhan@tongji.edu.cn}



\thanks{This work is supported by National Natural Science Foundation of China (Grant Nos. 11801363, 11771279 and 11671247).}

\begin{abstract}
In this paper, we study a class of non-weight modules over the affine-Virasoro algebra of type $A_1$,  which are free modules of rank one  when restricted to the Cartan subalgebra (modulo center).  We give  the classification of such modules.  Moreover, the simplicity and the isomorphism classes of these modules are determined.

\end{abstract}

\maketitle

\section{Introduction}

The Virasoro algebra is an infinite dimensional Lie algebra over $\C$ with basis $\{d_i, C\mid i\in\Z\}$ and defining relations
\begin{equation*}
[d_i,d_j]=(j-i)d_{i+j}+\delta_{i+j,0}\frac{i^3-i}{12}C,\quad
[d_i,C]=0,\,\,\,\forall\,i,j\in\mathbb{Z},
\end{equation*}
which is the universal central extension of the so-called infinite dimensional Witt algebra of rank one. The Virasoro algebra occurs as the algebra of the conformal group in one dimension, or in the form of two commuting copies. Affine Lie algebras play an important role in string theory and two-dimensional conformal field theory. It is well known that the Virasoro algebra and the affine Lie algebras have been widely used in many physics areas and mathematical branches. Their close relationship strongly suggests that they should be considered simultaneously, that is, as one algebraic structure. Actually this has led to the definition of the so-called affine-Virasoro algebra, which is the semidirect product of the Virasoro algebra and an affine Kac-Moody Lie algebra with a common center. Affine-Virasoro algebras  are very meaningful in the sense that they are closely connected to the conformal field theory. For example, the even part of $N=3$ superconformal algebra \cite{CL} is just the affine-Virasoro algebra of type $A_1$. Highest weight representations and integrable  representations of the affine-Virasoro algebras have been  extensively studied  (cf.~\cite{EJ}, \cite{JY}-\cite{LQ},  \cite{XH}). The author in \cite{B} presented the  classification of  all simple Harish-Chandra modules with nonzero central actions over the affine-Virasoro algebras. However, all simple uniform bounded modules over these algebras are not yet classified except the affine-Virasoro algebra of type $A_1$ \cite{GHL}.

In recent years, many authors constructed various simple non-Harish-Chandra modules and simple non-weight modules (cf.~\cite{BM}, \cite{CTZ}-\cite{CY}, \cite{LGZ}-\cite{TZ}). In particular, J. Nilsson \cite{N} constructed  a class of  $\mathfrak{sl_{n+1}}$-modules that are free of rank one when restricted to the Cartan subalgebra.  This kind of non-weight modules, which many authors call $U(\mathfrak{h})$-free modules, have been extensively studied. In the paper \cite{N} and a subsequent paper \cite{N1}, J. Nilsson showed that a finite dimensional simple Lie algebra has nontrivial $U(\mathfrak{h})$-free modules if and only if it is of type $A$ or $C$. Furthermore, the $U(\mathfrak{h})$-free modules of rank one for the Kac-Moody Lie algebras were determined in \cite{CTZ}. And the idea  was  exploited and generalized to consider modules over infinite  dimensional Lie algebras, such as the Witt algebras of arbitrary rank \cite{TZ},  Heisenberg-Virasoro algebra and $W(2,2)$ algebra  \cite{CG}, the algebras $\mathcal{V}ir(a,b)$  \cite{HCS}, the Lie algebras related to the Virasoro algebra \cite{CC,CY} and so on. The aim of this paper is to classify such modules for the the affine-Virasoro algebra of type $A_1$.

This paper is organized as follows. In Section 2, we construct a class of non-weight modules over the affine-Virasoro algebra of type $A_1$, and study the simplicity and isomorphic relations of these modules. Section 3 is devoted to classifying all modules  whose restriction to the Cartan subalgebra (modulo center) are free of rank one over the affine-Virasoro algebra of type $A_1$.

Throughout the paper, we denote by $\C ,\,\Z,\,\C^*,\,\Z_+$ the sets of complex numbers, integers, nonzero complex numbers and nonnegative integers, respectively. All vector spaces  are assumed to be  over $\C$. For a Lie algebra $\mathfrak{g}$, we use $U(\mathfrak{g})$ to denote the universal enveloping algebra of $\mathfrak{g}$.

\section{Preliminaries}
In this section, we shall introduce some basic notations and establish some related results.
\begin{defi}\label{defi2.1}\rm Let $L$ be a finite-dimensional Lie algebra with a non-degenerated invariant symmetric bilinear form $(\cdot, \cdot)$. The affine-Virasoro algebra $L_{av}$ is the vector space
$$L_{av}=L\otimes\C[t,t^{-1}]\oplus\C C\oplus \bigoplus_{i\in\Z}\C d_i$$
with the Lie brackets:
\begin{equation*}
\aligned
&[x\otimes t^m,y\otimes t^n]=[x,y]\otimes t^{m+n}+m(x,y)\delta_{m+n,0}C,\\
&[d_i,d_j]=(j-i)d_{i+j}+\delta_{i+j,0}\frac{i^3-i}{12}C,\\
&[d_i,x\otimes t^m]=mx\otimes t^{m+i},\quad [C, L_{av}]=0,
\endaligned
\end{equation*}where $x,y\in L$, $m,n,i,j\in\Z$ (if $L$ has no such form, we set $(x,y)=0$ for all $x,y\in L$).
\end{defi}
We see that if $L=\C e$ is one-dimensional, then $L_{av}$ is just the twisted Heisenberg-Virasoro algebra (one center element). In the following  we only consider specially $L$ as the simple Lie algebra $\mathfrak{sl_2}={\rm span_{\C}}\{e,\,f,\,h\}$.

\begin{defi} \rm Take $L$ to be the Lie algebra $\mathfrak{sl}_2$ in Definition \ref{defi2.1}.  Then the resulting
Lie algebra $\mathfrak L$ (here the bilinear form $(\cdot,\cdot)$ is normalized by $(e, f)=1$)   with $\C$-basis $\{e_i,\, f_i,\,h_i,\, d_i,\,C\mid i\in\Z\}$  subject to the following Lie brackets:
\begin{equation*}\label{L-action}
\aligned
&[e_i,f_j]=h_{i+j}+i\delta_{i+j,0}C,\\
&[h_i,e_j]=2e_{i+j},\quad [h_i,f_j]=-2f_{i+j},\\
&[d_i,d_j]=(j-i)d_{i+j}+\delta_{i+j,0}\frac{i^3-i}{12}C,\\
&[d_i,h_j]=jh_{i+j},\quad [h_i,h_j]=-2i\delta_{i+j,0}C,\\
&[d_i,e_j]=je_{i+j},\quad [d_i,f_j]=jf_{i+j},\\
&[e_i,e_j]=[f_i,f_j]=[C,\mathfrak{L}]=0,
\endaligned
\end{equation*} is called the affine-Virasoro algebra of type $A_1$.

\end{defi}

\begin{defi}\label{defi2.2}\rm Let $\C[s,t]$ be the polynomial algebra in variables $s$ and $t$ with coefficients in $\C$. For $\lambda,\a\in\C^*, \b,\r\in\C, i\in\Z$ and $g(s,t)\in\C[s,t]$, define the action of $\mathfrak{L}$ on $\C[s,t]$ as follows:
\begin{align*}
\Omega(\lambda,\a, \b, \r):&\ \ \ \ e_i\cdot g(s,t)=\lambda^i\a g(s-i,t-2),\\
& \ \ \ \ f_i\cdot g(s,t)=-\frac{\lambda^i}\a(\frac{t}{2}-\b)(\frac{t}{2}+\b+1)g(s-i,t+2),\\
& \ \ \ \ h_i\cdot g(s,t)=\lambda^itg(s-i,t),\ \ \ \ d_i\cdot g(s,t)=\lambda^i(s+i\r)g(s-i,t),\\
& \ \ \ \ C\cdot g(s,t)=0;\\
\Delta(\lambda, \a, \b, \r):&\ \ \ \ e_i\cdot g(s,t)=-\frac{\lambda^i}{\a}(\frac{t}{2}+\b)(\frac{t}{2}-\b-1)g(s-i,t-2),\\
& \ \ \ \ f_i\cdot g(s,t)=\lambda^i\a g(s-i,t+2), \ \ \ \ h_i\cdot g(s,t)=\lambda^{i}tg(s-i,t),\\
& \ \ \ \ d_i\cdot g(s,t)=\lambda^i(s+i\r)g(s-i,t),\\
& \ \ \ \ C\cdot g(s,t)=0;\\
\Theta(\lambda, \a, \b, \r):&\ \ \ \ e_i\cdot g(s,t)=\lambda^i\a(\frac{t}{2}+\b)g(s-i,t-2),\\
& \ \ \ \ f_i\cdot g(s,t)=-\frac{\lambda^i}{\a}(\frac{t}{2}-\b)g(s-i,t+2),\\
& \ \ \ \ h_i\cdot g(s,t)={\lambda^i}tg(s-i,t),\ \ \ \ d_i\cdot g(s,t)=\lambda^i(s+i\r)g(s-i,t),\\
& \ \ \ \ C\cdot g(s,t)=0.
\end{align*}
\end{defi}

\begin{rema}\label{r-mk1}\rm  (1) Whenever  we consider the action of $\mathfrak L$ on $\C[s,t]$, we always mean one of these above.

(2) Denote by $\mathfrak{H}_{h}$ the vector space spanned by  the set $\{h_i,d_i,C\mid i\in\Z\}.$  An important fact needs to be pointed here is: though $\mathfrak H_h$ is an quotient algebra of the Heisenberg-Virasoro algebra $\mathcal Vir(0,0),$ they have the same submodule structure on $\C[s,t]$ (cf. \cite{CG} or \cite{HCS}).

\end{rema}

\begin{prop}
$\Omega(\lambda,\a, \b, \r),\Delta(\lambda, \a, \b, \r)$ and  $\Theta(\lambda, \a, \b, \r)$ are $\mathfrak{L}$-modules under the actions given in Definition \ref{defi2.2}. Moreover$,$ $\Omega(\lambda,\a, \b, \r)$ and $\Delta(\lambda, \a, \b, \r)$ are simple for all $\lambda,\a\in\C^*$ and $\b,\r\in\C;$ $\Theta(\lambda, \a, \b, \r)$ is simple if and only if $2\b \notin \Z_+$.
\end{prop}
\begin{proof}
For the first statement, we only tackle  the case  $\Omega(\lambda,\a, \b, \r)$, since the other two cases can be treated similarly.  In view of the  $\mathfrak{H}_{h}$-action, we know the following relations
\begin{equation*}
\aligned
&d_i\circ d_j-d_j\circ d_i=[d_i,d_j],\\
&d_i \circ h_j-h_j\circ d_i=[d_i,h_j],\\
&h_i\circ h_j-h_j\circ h_i=[h_i,h_j]
\endaligned
\end{equation*} hold on $\Omega(\lambda,\a, \b, \r)$ by \cite[Proposition 2.2]{HCS}.
Note according to the above definition for any $g(s,t)\in \C[s,t]$ that
\begin{eqnarray*}
&&e_i\cdot f_j\cdot g(s,t)-f_j\cdot e_i\cdot g(s,t)\nonumber\\
&=&e_i\cdot\big(-\frac{\lambda^j}{\a}(\frac{t}{2}-\b)(\frac{t}{2}+\b+1)g(s-j,t+2)\big)\nonumber\\
&&-f_j\cdot\big(\lambda^i\a g(s-i,t-2)\big)\nonumber\\
&=&-\lambda^{i+j}(\frac{t}{2}-\b-1)(\frac{t}{2}+\b)g(s-i-j,t)\nonumber\\
&&+\lambda^{i+j}(\frac{t}{2}-\b)(\frac{t}{2}+\b+1)g(s-i-j,t)\nonumber\\
&=&\lambda^{i+j}tg(s-i-j,t)\nonumber\\
&=&\big(h_{i+j}+i\delta_{i+j,0}C\big)\cdot g(s,t)=[e_i,f_j]\cdot g(s,t),
\end{eqnarray*}
\begin{eqnarray*}
&&h_i\cdot e_j\cdot g(s,t)-e_j\cdot h_i\cdot g(s,t)\nonumber\\
&=&h_i\cdot\big(\lambda^{j}\a g(s-j,t-2)\big)-e_j\cdot\big(\lambda^{i}t g(s-i,t)\big)\nonumber\\
&=&\lambda^{i+j}\a tg(s-i-j,t-2)-\lambda^{i+j}\a (t-2)g(s-i-j,t-2)\nonumber\\
&=&2\lambda^{i+j}\a g(s-i-j,t-2)\nonumber\\
&=&2e_{i+j}\cdot g(s,t)=[h_i,e_j]\cdot g(s,t),
\end{eqnarray*}
\begin{eqnarray*}
&&h_i\cdot f_j\cdot g(s,t)-f_j\cdot h_i\cdot g(s,t)\nonumber\\
&=&h_i\cdot\big(-\frac{\lambda^{j}}{\a}(\frac{t}{2}-\b)(\frac{t}{2}+\b+1) g(s-j,t+2)\big)\nonumber\\
&&-f_j\cdot\big(\lambda^{i}tg(s-i,t)\big)\nonumber\\
&=&-\frac{\lambda^{i+j}}{\a}t(\frac{t}{2}-\b)(\frac{t}{2}+\b+1) g(s-i-j,t+2)\nonumber\\
&&+\frac{\lambda^{i+j}}{\a}(\frac{t}{2}-\b)(\frac{t}{2}+\b+1)(t+2)g(s-i-j,t+2)\nonumber\\
&=&\frac{2\lambda^{i+j}}{\a}(\frac{t}{2}-\b)(\frac{t}{2}+\b+1)g(s-i-j,t+2)\nonumber\\
&=&-2f_{i+j}\cdot g(s,t)=[h_i,f_j]\cdot g(s,t),
\end{eqnarray*}
\begin{eqnarray*}
&&d_i\cdot e_j\cdot g(s,t)-e_j\cdot d_i\cdot g(s,t)\nonumber\\
&=&d_i\cdot\big(\lambda^j\a g(s-j,t-2)\big)-e_j\cdot\big(\lambda^i(s+i\r)g(s-i,t)\big)\nonumber\\
&=&\lambda^{i+j}\a(s+i\r)g(s-i-j,t-2)-\lambda^{i+j}\a(s+i\r-j)g(s-i-j,t-2)\nonumber\\
&=&j\lambda^{i+j}\a g(s-i-j,t-2)\nonumber\\
&=&je_{i+j}\cdot g(s,t)=[d_i,e_j]\cdot g(s,t)
\end{eqnarray*}
and
\begin{eqnarray*}
&&d_i\cdot f_j\cdot g(s,t)-f_j\cdot d_i\cdot g(s,t)\nonumber\\
&=&d_i\cdot\big(-\frac{\lambda^j}{\a}(\frac{t}{2}-\b)(\frac{t}{2}+\b+1)  g(s-j,t+2)\big)\nonumber\\
&&-f_j\cdot\big(\lambda^i(s+i\r)g(s-i,t)\big)\nonumber\\
&=&-\frac{\lambda^{i+j}}{\a}(s+i\r)(\frac{t}{2}-\b)(\frac{t}{2}+\b+1)g(s-i-j,t+2)\nonumber\\
&&+\frac{\lambda^{i+j}}{\a}(s+i\r-j)(\frac{t}{2}-\b)(\frac{t}{2}+\b+1)g(s-i-j,t+2)\nonumber\\
&=&-\frac{j\lambda^{i+j}}{\a}(\frac{t}{2}-\b)(\frac{t}{2}+\b+1) g(s-i-j,t+2)\nonumber\\
&=&jf_{i+j}\cdot g(s,t)=[d_i,f_j]\cdot g(s,t).
\end{eqnarray*}
And the last three  relations
\begin{equation*}
\aligned
&e_i\circ e_j-e_j\circ e_i=[e_i,e_j],\\
 &f_i\circ f_j-f_j\circ f_i=[f_i,f_j],\\
 {\rm and}\quad & C\circ x-x\circ C=0, \forall x\in \mathcal L
\endaligned
\end{equation*}can be checked easily, proving the first statement.

 Note that  $d_0$ and  $h_0$ are in fact the left multiplication operators $s$ and $t$  on $\C[s,t].$ In particular, $1$ is a generator of the $\mathfrak L$-module $\C[s,t]$.  So    the simplicity of these modules is equivalent to determining whether every nonzero $\mathfrak L$-submodule $\C[s,t]$ contains $1$. Let $W$ be a nonzero submodule of $\Omega(\lambda,\a, \b, \r),$ $\Delta(\lambda, \a, \b, \r)$ or $\Theta(\lambda, \a, \b, \r)$. Then regarding as an $\mathfrak{H}_{h}$-submodule,  either $W=\C[s,t]g(t)$ or $W=\C[s,t]sg(t)+\C[s,t]tg(t)$ for some nonzero polynomial $g(t)\in\C[t]$ by \cite[Theorem 2.3]{H} and Remark \ref{r-mk1}(2).
\begin{clai}\rm  $1\in W$ if $W\subseteq\Omega(\lambda,\a, \b, \r)$\ (resp. $\Delta(\lambda, \a, \b, \r)$). \end{clai}
Consider $ W\subseteq\Omega(\lambda,\a, \b, \r)$.
If $W=\C[s,t]g(t)$, then from the definition of the module structure one can inductively show that
$$g(t-2k)=\a^{-k}e_0^k\cdot g(t)\in W.$$
Note that we can make $g(t)$ and $g(t-2k)$ coprime by choosing $k$ large enough.

If  $W=\C[s,t]sg(t)+\C[s,t]tg(t)$,  then using induction on $k$ one has
\begin{eqnarray*}
\label{eqn1}&(\lambda e_0)^k\cdot sg(t)=(\lambda\a)^{k}sg(t-2k), \\
\label{eqn2}&e_1^k\cdot sg(t)=(\lambda\a)^{k}(s-k)g(t-2k),
\end{eqnarray*}
 and  the subtraction of these two gives
 $$g(t-2k)=\frac{1}{k(\lambda\a)^{k}}\big((\lambda e_0)^{k}\cdot sg(t)-e_1^{k}\cdot sg(t)\big)\in W.$$
It is easy to see that $tg(t)$ and $g(t-2k)$ coprime by choosing $k$ large enough. That is, we have coprime elements in $W$ and therefore $1\in W$ in both cases. Using the similar argument as  for $\Omega(\lambda,\a, \b, \r)$ but respectively replacing $e_0, e_1$ with $f_0, f_1$ we see that $1\in W$ if $W\subseteq\Delta(\lambda, \a, \b, \r)$.

Now assume that $W\subseteq\Theta(\lambda, \a, \b, \r).$
\begin{clai} \rm $1\in W$ if   $2\beta\notin\Z_+$ and there exists a nonzero (simple) $\mathcal L$-submodule $V$ of  $\Theta(\lambda, \a, \b, \r)$ such that $1\notin V$ if $2\b\in\Z_+$. \end{clai}

If $W=\C[s,t]g(t)$, then using induction on $k$ one  immediately has
\begin{eqnarray*}
\label{eqn3}&e_0^k\cdot g(t)=\a^{k}\prod_{n=0}^{k-1}(\frac{t}{2}+\b-n)g(t-2k), \\
\label{eqn4}&f_0^k\cdot g(t)=(-\frac{1}{\a})^{k}\prod_{n=0}^{k-1}(\frac{t}{2}-\b+n)g(t+2k).
\end{eqnarray*}
And  for the other case  $W=\C[s,t]sg(t)+\C[s,t]tg(t)$,   one has
\begin{eqnarray*}
\label{eqn5}&e_0^k\cdot tg(t)=\a^{k}\prod_{n=0}^{k-1}(\frac{t}{2}+\b-n)(t-2k)g(t-2k),\\
\label{eqn6}&f_0^k\cdot tg(t)=(-\frac{1}{\a})^{k}\prod_{n=0}^{k-1}(\frac{t}{2}-\b+n)(t+2k)g(t+2k).
\end{eqnarray*}
Note that if $2\b\notin\Z_+$,   then $e_0^k\cdot g(t)$ and $f_0^k\cdot g(t)$ (resp. $e_0^k\cdot tg(t)$ and $f_0^k\cdot tg(t)$) for large enough $k$ are coprime elements in $W$.  Thus,  $1\in W$.

Assume that $2\b\in\Z_+$. Consider the vector space
$$V:=\C[s,t]\prod_{n=0}^{2\b}(\frac{t}{2}+\b-n).$$  Clearly,
\begin{equation}\label{ffd}
\mathfrak{H}_{h}\cdot V\subset V. \end{equation}
For any $0\neq g(s,t)\prod_{n=0}^{2\b}(\frac{t}{2}+\b-n)\in V$ and $i\in\Z$, we have
\begin{eqnarray}\label{asf}e_i \cdot g(s,t)\prod_{n=0}^{2\b}(\frac{t}{2}+\b-n)&=&\lambda^i\a(\frac{t}{2}+\b)g(s-i,t-2)\prod_{n=0}^{2\b}(\frac{t}{2}+\b-n-1)\nonumber\\ &=&\lambda^i\a(\frac{t}{2}-\b-1)g(s-i,t-2)\prod_{n=0}^{2\b}(\frac{t}{2}+\b-n)\end{eqnarray}
and
\begin{eqnarray}\label{asf2}f_i \cdot g(s,t)\prod_{n=0}^{2\b}(\frac{t}{2}+\b-n)&=&-\frac{\lambda^i}{\a}(\frac{t}{2}-\b)g(s-i,t+2)\prod_{n=0}^{2\b}(\frac{t}{2}+\b-n+1)\nonumber\\ &=&-\frac{\lambda^i}{\a}(\frac{t}{2}+\b+1)g(s-i,t+2)\prod_{n=0}^{2\b}(\frac{t}{2}+\b-n),\end{eqnarray}
These along with \eqref{ffd} show that $V$ is a proper submodule. Furthermore, \eqref{asf} and \eqref{asf2} entail us to establish an $\mathfrak{L}$-module isomorphism \begin{equation*}
\tau: \Theta(\lambda, \a, -\b-1, \r)\to V\ \ \ \ g(s,t)\mapsto g(s,t)\prod_{n=0}^{2\b}(\frac{t}{2}+\b-n).\end{equation*}
Then   $V\cong  \Theta(\lambda, \a, -\b-1, \r)$ is simple,  since $2(-\b-1)\notin\Z_+$. This completes the proof.
\end{proof}

The following proposition gives a characterization of  isomorphisms between the  $\mathfrak{L}$-modules constructed above.
\begin{prop}
Let $\lambda, \lambda^{\prime}, \a,\a^{\prime}\in\mathbb{C}^*,\b, \b^{\prime},\r, \r^{\prime}\in\mathbb{C}$. Then
\begin{eqnarray} \label{xxit11}&\Omega(\lambda,\a,\b,\r)\cong
\Omega(\lambda^{\prime},\a^{\prime},\b^{\prime},\r^{\prime})\Longleftrightarrow (\lambda,\a,\b,\r)=(\lambda^{\prime},\a^{\prime},\b^{\prime},\r^{\prime})\\
&\quad{\rm or}\quad (\lambda,\a,\b,\r)=(\lambda^{\prime},\a^{\prime},-\b^{\prime}-1,\r^{\prime});\nonumber\\
\label{xxit1}&\Delta(\lambda,\a,\b,\r)\cong
\Delta(\lambda^{\prime},\a^{\prime},\b^{\prime},\r^{\prime})\Longleftrightarrow (\lambda,\a,\b,\r)=(\lambda^{\prime},\a^{\prime},\b^{\prime},\r^{\prime})\\
&\quad{\rm or}\quad (\lambda,\a,\b,\r)=(\lambda^{\prime},\a^{\prime},-\b^{\prime}-1,\r^{\prime});\nonumber\\
\label{xxit2}&\Theta(\lambda,\a,\b,\r)\cong\Theta(\lambda^{\prime},\a^{\prime},\b^{\prime}\r^{\prime})\Longleftrightarrow (\lambda,\a,\b,\r)=(\lambda^{\prime},\a^{\prime},\b^{\prime},\r^{\prime}).\end{eqnarray}
\end{prop}
\begin{proof} We only prove  \eqref{xxit11}, a similar argument can be applied to \eqref{xxit1} and \eqref{xxit2}. For this, it suffices to show the $``\Longrightarrow"$ part.
Let $\varphi:\Omega(\lambda,\a,\b,\r)\to\Omega(\lambda^{\prime},\a^{\prime},\b^{\prime},\r^{\prime})$ be an  isomorphism of $\mathfrak{L}$-modules. Viewing $\Omega(\lambda,\a,\b,\r)$ and $\Omega(\lambda^{\prime},\a^{\prime},\b^{\prime},\r^{\prime})$ as $\mathfrak{H}_{h}$-modules, we get $(\lambda,\r)=(\lambda^{\prime},\r^{\prime})$ by \cite[Proposition 2.3(ii)]{HCS}. Now for any  $g(t)\in\C[t]$,  we have \begin{eqnarray*}\label{xx}&\varphi(g(t))=\varphi(g(h_{0})\cdot1)=g(h_0)\cdot\varphi(1)=g(t)\varphi(1),\\
&\varphi(\a)=\varphi(e_0\cdot1)=e_0\cdot \varphi(1)=\a^{\prime}\varphi(1)\\
&{\rm and}\quad \varphi(-\frac{1}{\a}(\frac{t}{2}-\b)(\frac{t}{2}+\b+1))
=\varphi(f_0\cdot1)=f_0\cdot\varphi(1)=-\frac{1}{\a^{\prime}}(\frac{t}{2}-\b^{\prime})(\frac{t}{2}+\b^{\prime}+1)\varphi(1).\end{eqnarray*}
It is easy to see from the first two formulae above that  $\a=\a^{\prime}$, which together with the first and third formulae gives rise to $\b=\b^{\prime}$ or $\b=-\b^{\prime}-1$. This completes the proof.
\end{proof}

\section{Main result}
It is clear that  the Cartan subalgebra  (modulo center) of $\mathfrak{L}$ is spanned by $h_0$ and $d_0$. The main result of the present paper is to classify all modules over $\mathfrak{L}$ whose restrictions to $U(\C d_0 \oplus\C h_0)$ are
free of rank $1$. Before presenting the main result, we first give a lemma, which  can be easily shown  by induction on $m$.
\begin{lemm}\label{lemma1}
For any $i\in\Z$ and $0\le m\in\Z,$ we have
\begin{eqnarray*}
&e_i d_0^m=(d_0-i)^me_i,\ \ \ \ f_i d_0^m=(d_0-i)^mf_i,\\
&e_i h_0^m=(h_0-2)^me_i,\ \ \ \ f_i h_0^m=(h_0+2)^mf_i.
\end{eqnarray*}
\end{lemm}
\begin{theo}\label{theo1}Any $U(\mathfrak{L})$-module $M$  such that its restriction to $U(\C d_0 \oplus\C h_0)$ is free of rank $1$ is isomorphic to one of the modules $$\Omega(\lambda,\a, \b, \r),\,\Delta(\lambda, \a, \b, \r),\,\Theta(\lambda, \a, \b, \r),$$ for some $\lambda,\a\in\C^*$ and $\b,\r\in\C$.
\end{theo}

\begin{proof} Let $M$ be an $\mathfrak{L}$-module which is a free  $U(\C d_0 \oplus\C h_0)$-module of
 rank $1$. Then $M=U(\C d_0 \oplus\C h_0)$. It follows from viewing as an $\mathfrak{H}_{h}$-module that $M\cong\Phi(\lambda,0,\mathbf {p})$ by \cite[Theorem 3.1]{HCS}:
\begin{equation*}
\aligned
&d_i\cdot g(d_0,h_0)=\lambda^i\big(d_0+p_i (h_0)\big)g(d_0-i,h_0),\\
&h_i\cdot g(d_0,h_0)=\lambda^ih_0g(d_0-i,h_0),\ \ \ \ C\cdot g(d_0,h_0)=0,
\endaligned
\end{equation*}
where $g(d_0,h_0)\in U(\C d_0 \oplus\C h_0),\,\lambda\in\C^*,\, i\in\Z$ and $$\mathbf{p}=\big(p_i(h_0)\big)_{i\in\Z}\in\left\{\big(p_i(h_0)\big)_{i\in\Z}\mid p_i(h_0)=\sum_{l=0}^{\infty}p^{(l)}ih_{0}^{l}\in\C[h_0], p^{(l)}\in\C\right\}.$$

For any $i\in\Z$, assume that $E_i(d_0,h_0)=e_i\cdot1$ and  $F_i(d_0,h_0)=f_i\cdot1$ for some $E_i(d_0,h_0)$, $F_i(d_0,h_0)\in U(\C d_0\oplus \C h_0)$.  Take any
$$g(d_0, h_0)=\sum_{j,k\in\Z_+}g_{j,k}d_{0}^j h_{0}^k\in U(\C d_0\oplus \C h_0),\quad {\rm where}\ g_{j,k}\in\C.$$ Then by  Lemma \ref{lemma1},
\begin{eqnarray}\label{ww8}
e_i\cdot g(d_0, h_0)&=&e_i\cdot \sum_{j,k\in\Z_+}g_{j,k}d_{0}^j h_{0}^k=
\sum_{j,k\in\Z_+}g_{j,k}(d_{0}-i)^{j} e_i\cdot h_{0}^{k}\nonumber\\
&=&
\sum_{j,k\in\Z_+}g_{j,k}(d_{0}-i)^{j} (h_{0}-2)^{k}E_i(d_0,h_0)
\end{eqnarray}
and
\begin{eqnarray}\label{ww9}
f_i\cdot g(d_0, h_0)&=&f_i\cdot \sum_{j,k\in\Z_+}g_{j,k}d_{0}^j h_{0}^k=
\sum_{j,k\in\Z_+}g_{j,k}(d_{0}-i)^{j} f_i\cdot h_{0}^{k}\nonumber\\
&=&
\sum_{j,k\in\Z_+}g_{j,k}(d_{0}-i)^{j} (h_{0}+2)^{k}F_i(d_0,h_0).
\end{eqnarray} Thus, the actions of $e_i$ and $f_i$ on $M$ are completely determined by $E_i(d_0,h_0)$ and  $F_i(d_0,h_0)$, respectively. For this reason, in what follows we only need to determine $E_i(d_0,h_0)$ and $F_i(d_0,h_0)$ for all $i\in\Z$.

Using \eqref{ww8} and \eqref{ww9}, we present some formulae here, which will be used to do calculations in the following.
The equations $$[e_0,f_0]\cdot1=h_0\cdot1,[e_0,e_1]\cdot1=0,[f_0,f_1]\cdot1=0$$
are respectively equivalent to
\begin{eqnarray}
\label{equ1}&E_0(d_0,h_0)F_0(d_0,h_0-2)-E_0(d_0,h_0+2)F_0(d_0,h_0)=h_0, \\
\label{equ2}&E_0(d_0,h_0)E_1(d_0,h_0-2)=E_0(d_0-1,h_0-2)E_1(d_0,h_0),\\
\label{equ3}&F_0(d_0,h_0)F_1(d_0,h_0+2)=F_0(d_0-1,h_0+2)F_1(d_0,h_0).
\end{eqnarray}
Note that $E_0(d_0,h_0)F_0(d_0,h_0)\neq0$. Since otherwise $E_0(d_0,h_0)F_0(d_0,h_0)=0$ and \eqref{equ1} would give $h_0=0$, which is absurd. So we may assume
\begin{eqnarray*}
E_0(d_0,h_0)=\sum_{i=0}^{m}a_{i}(d_0)h_0^i\quad {\rm and }\quad F_0(d_0,h_0)=\sum_{i=0}^{n}b_{i}(d_0)h_0^i
\end{eqnarray*}
for some $a_{i}(d_0),b_{i}(d_0)\in\C[d_0]$ and $a_{m}(d_0)b_{n}(d_0)\neq 0$. Inserting  these expressions into \eqref{equ1} yields
\begin{eqnarray*}\label{eq1}
\sum_{i=0}^{m}a_{i}(d_0)h_0^i\sum_{i=0}^{n}b_{i}(d_0)(h_0-2)^i
-\sum_{i=0}^{m}a_{i}(d_0)(h_0+2)^i\sum_{i=0}^{n}b_{i}(d_0)h_0^i=h_0,
\end{eqnarray*}
comparing  highest degree terms, with respect to $h_0$,  of both sides of  which   gives
\begin{eqnarray*}\label{ass}
m+n=2\quad {\rm and}\quad a_m(d_0)b_n(d_0)=-\frac{1}{4}.
\end{eqnarray*}From now on the discussion are divided into the following  three cases.
\begin{case}\label{c-a-e-1}$(m,n)=(0,2)$. \end{case}

In this case, we have
\begin{eqnarray*}\label{e0}E_0(d_0,h_0)=\a\quad {\rm and}\quad F_0(d_0,h_0)=-\frac{1}{4\a}h_0^2+u_1(d_0)h_0+v_1(d_0)
\end{eqnarray*}
for some $\a\in\C^{*}$ and $u_1(d_0),v_1(d_0)\in\C[d_0]$. Inserting  the two expressions  into \eqref{equ1} one has $u_1(d_0)=-\frac{1}{2\a}$. Namely,
\begin{eqnarray}\label{ee}F_0(d_0,h_0)=-\frac{1}{4\a}h_0^2-\frac{1}{2\a}h_0+v_1(d_0).\end{eqnarray}
It follows from this and  $[h_1,f_0]\cdot1=-2f_1\cdot1$ that
\begin{eqnarray}\label{ass1}
F_1(d_0,h_0)=\lambda\big(-\frac{1}{4\a}h_0^2-\frac{1}{2\a}h_0+\frac{1}{2}h_0v_1(d_0)-\frac{1}{2}h_0v_1(d_0-1)+v_1(d_0)\big).
\end{eqnarray}
Inserting  \eqref{ee}, \eqref{ass1} into  \eqref{equ3} and equating the terms  do not depend on $h_0$ of both sides, we obtain

$$v_1(d_0)^2=v_1(d_0)v_1(d_0-1),$$ which forces $v_1(d_0)\in\C$. So
\begin{equation*}F_0(d_0,h_0)=-\frac{1}{\a}(\frac{h_0}{2}-\b)(\frac{h_0}{2}+\b+1)\quad {\rm for\ some}\ \b\in\C.
\end{equation*}
Then $$E_i(d_0,h_0)=e_i\cdot1=\frac12[h_i,e_0]\cdot1=\frac12(\alpha h_i\cdot1-\lambda^i e_0\cdot h_0\cdot1)=\frac{1}2\big(\alpha\lambda^ih_0-\alpha\lambda^i(h_0-2)\big)=\alpha\lambda^i,$$
 \begin{eqnarray*}
F_i(d_0,h_0)&=&f_i\cdot1=\frac12[f_0,h_i]\cdot1=\frac12(\lambda^if_0\cdot h_0\cdot1-h_i\cdot F_0(d_0,h_0))\\
&=&\frac{\lambda^i}2\big((h_0+2)F_0(d_0,h_0)-h_0F_0(d_0-i,h_0)\big)\\
&=&-\frac{\lambda^i}{\a}(\frac{h_0}{2}-\b)(\frac{h_0}{2}+\b+1),\ \ \ \ \forall i\in\Z.
\end{eqnarray*}
And  from   $[d_i,e_0]\cdot1=0$  we can deduce that  $p_i(h_0)\in\C$, i.e., $p_i(h_0)=i\gamma$ for some $\r\in\C$ and all $i\in\Z$. Thus, $M\cong\Omega(\lambda,\a, \b, \r)$.
\begin{case} $(m,n)=(2,0)$. \end{case}
Interchanging $E_0(d_0,h_0)$ and $F_0(d_0, h_0)$  and then imitating the proof of Case \ref{c-a-e-1},  we will see that  $M\cong\Delta(\lambda,\a, \b, \r).$

\begin{case} $(m,n)=(1,1)$. \end{case}
Now we can assume that
\begin{eqnarray}\label{assi}E_0(d_0,h_0)=\frac{\a}{2}h_0+u_2(d_0)\end{eqnarray}
and \begin{eqnarray*}\label{assu}F_0(d_0,h_0)=-\frac{1}{2\a} h_0+v_2(d_0)\end{eqnarray*}
for some $\a\in\C^{*}$ and $u_2(d_0),v_2(d_0)\in\C[d_0]$. Inserting the two expressions into \eqref{equ1} forces $$u_2(d_0)=\a^2v_2(d_0).$$ It follows from  \eqref{assi} and $[h_1,e_0]\cdot1=2e_1\cdot1$ that
\begin{eqnarray}\label{assg}
E_1(d_0,h_0)=\lambda\big(\frac{\a}{2}h_0+\frac{1}{2}h_0u_2(d_0-1)-\frac{1}{2}h_0u_2(d_0)+u_2(d_0)\big).
\end{eqnarray}
Inserting   \eqref{assi}, \eqref{assg}  into  \eqref{equ2} and then equating the terms  do not depend on $h_0$ of both sides, we have $$u_2(d_0)^2=u_2(d_0)u_2(d_0-1),$$ which implies $u_2(d_0)\in\C$. Thus,
\begin{equation*}E_0(d_0,h_0)=\a(\frac{h_0}{2}+\b) \quad {\rm and}\quad F_0(d_0,h_0)=-\frac{1}{\a}(\frac{h_0}{2}-\b)\quad {\rm for\ some}\ \b\in\C.
\end{equation*}
These along with $$[h_i,e_0]\cdot1=2e_i\cdot1\quad {\rm and}\quad [h_i,f_0]\cdot1=-2f_i\cdot1,\ \ \ \ \forall i\in\Z$$ give
$$E_{i}(d_0,h_0)=\lambda^i\a(\frac{h_0}{2}+\b)\quad {\rm and}\quad F_{i}(d_0,h_0)=-\frac{\lambda^i}{\a}(\frac{h_0}{2}-\b),\ \ \ \ \forall i\in\Z.$$
And from $[d_i,e_0]\cdot1=0$ we see  that for any $i\in\Z$, $p_i(h_0)=i\r$ for some $\r\in\C$. Then in this case, $M\cong\Theta(\lambda,\a, \b, \r)$.
\end{proof}


\end{document}